\RequirePackage{fix-cm}
\documentclass[smallextended]{svjour3}       
\smartqed  

\usepackage[utf8]{inputenc}
\usepackage[english]{babel}
\usepackage{listings}
\usepackage[ruled,vlined]{algorithm2e}
\usepackage{url}
\usepackage{graphicx,subfigure}
\usepackage{psfrag}
\usepackage[leqno]{amsmath,mathtools}
\usepackage{tabularx}
\usepackage{verbatim}
\usepackage{float}
\usepackage{amssymb}
\usepackage{latexsym}
\usepackage{pifont}

\usepackage{color}

\newcommand{\vect}{{\operatorname {vec}}}
\newcommand{\rank}{{\operatorname{rank}}}

\newcommand{\re}{{{\mathrm{Re}}~}}
\newcommand{\im}{{{\mathrm{Im}}~}}

\newcommand{\RR}{\mathbb{R}}

\newcommand{\diag}{\operatorname{diag}}
\newcommand{\kron}{\otimes}
\newcommand{\norm}[1]{||#1||}

\spnewtheorem{assumption}{Assumption}{\bfseries}{\itshape}
\begin{document}

\title{Implicit algorithms for eigenvector nonlinearities
}


\author{Elias Jarlebring         \and
        Parikshit Upadhyaya 
}


\institute{Elias Jarlebring \at
              Lindstedtsvägen 25,\\
              Department of Mathematics,\\
              SeRC - Swedish e-Science research center,\\
              Royal Institute of Technology, SE-11428, Stockholm, Sweden
              Tel.: +4687906694\\
              \email{eliasj@kth.se}           
           \and
           Parikshit Upadhyaya \at
              Lindstedtsvägen 25,\\
              Department of Mathematics,\\
              SeRC - Swedish e-Science research center,\\
              Royal Institute of Technology, SE-11428, Stockholm, Sweden
              Tel.: +4687908457\\
              \email{eliasj@kth.se}  
}


\maketitle

\begin{abstract}
We study and derive algorithms
for nonlinear eigenvalue problems, where the system
matrix depends on the eigenvector, or several eigenvectors
(or their corresponding invariant subspace).
The algorithms are derived from an implicit
viewpoint. More precisely, we change the Newton update equation
in a way that the next iterate does not only appear linearly
in the update equation.
Although, the modifications of the update equation
make the methods implicit we  show how corresponding
iterates can be computed explicitly.
Therefore we can carry out steps of the implicit
method using explicit procedures.
In several cases, these procedures involve a solution of standard eigenvalue problems. We propose two modifications, one of the modifications leads directly to a well-established method (the self-consistent field iteration)
whereas the other method is to our knowledge new and has several attractive properties. Convergence theory is provided along with several simulations which illustrate the properties of  the algorithms. 
\end{abstract}

\section{Introduction}\label{sec:intro}
Let $M\subset\mathbb{R}^{n\times n}$ denote the set of symmetric
$n\times n$-matrices. Let $A:\mathbb{R}^{n\times p}\to M$, $p<n$. We consider the problem of finding $V\in \mathbb{R}^{n\times p}$ and
a symmetric $S \in \mathbb{R}^{p\times p}$ such that
\begin{eqnarray}\label{eq:main}
A(V)V &= VS,\nonumber\\
V^TV &= I.
\end{eqnarray}
This is the general formulation of the
eigenvector-dependent nonlinear eigenvalue problem.
In our work $A$ satisfies  $A(VP)=A(V)$ for any
non-singular matrix $P$ such that the
range of $V$ can be 
seen as an invariant subspace of $A$. This property
(and a notion of invariant subspace)
is characterized in Section~\ref{sec:inv},
we show how a problem which does satisfy
the condition can be transformed 
where we also give a problem transformation 
applicable
when the condition is not satisfied.

The $V$-dependent matrix $A(V)$ is assumed to satisfy
certain properties such that, such that the columns
span of $V$ can be This notion of
invariant subspace is characterized in Section~\ref{sec:inv},
in particular conditions on $A$.

If $p=1$ the setting reduces to
a class of problem which has received some attention,
mostly in application specific settings. In this case
we need to determine $v\in\RR^{n}$ and $\lambda\in\RR$ such that
\begin{equation}\label{eq:main-p1}
A(v)v=\lambda v
\end{equation}
where $\|v\|=1$. Our results are, in particular, applicable for this case.

A number of algorithms have been proposed for the above problems, as we summarize below. In this paper we propose to derive algorithms based on implicit formulations, in particular
based on implicit improvements of Newton's method.
One proposed algorithm leads to a linearly convergent
well-established method, whereas the other approach
leads to a new method with quadratic convergence. Both
of the implicit approaches have advantageous properties
for certain problem classes that we characterize. 

Our approach is based on viewing iterative eigenvalue solvers
(for eigenvector nonlinearities)
as modifications of Newton's method.
This has also been done for standard eigenvalue problems,
already by Wilkinson and Peters \cite{Peters:1979:INVERSE}. See also the recent review paper \cite{Tapia:2018:INVNEWT} and the paper by Unger \cite{Unger:1950:NICHTLINEARE}.

One of the most important applications for \eqref{eq:main} is
within the field of quantum mechanics and electronic structure calculations.
Discretization methods in combination with the Hartree-Fock approximation
or the Kohn-Sham equations lead to
problems of type \eqref{eq:main}. See standard literature in quantum chemistry \cite{Szabo:1996:QC}. For a survey of numerical methods, see \cite{Saad:2010:ELECSTRUCT}. Considerable application specific research has been carried to specialized algorithms for this problem, mainly
based on the \emph{self-consistent field iteration} (SCF). SCF is an iterative method that involves solving a linear eigenvalue problem in each step until convergence or \textit{self-consistency}. 
The convergence of SCF and its variants has been studied in a number of works which can be classified into two broad categories: the optimization based approach of looking at \eqref{eq:main} as the optimality conditions of a minimization problem \cite{Cances:2000:SCF}, \cite{Levitt:2012:CONVERGENCE}, \cite{liu2015analysis}, \cite{liu2014convergence} or different matrix analysis based approaches \cite{Yang:2009:SCF}, \cite{ParikElias:2018:DENSITY}. Strategies for accelerating the convergence of SCF have also been studied well, e.g., \cite{Rohwedder:2011:DIIS}, \cite{Pulay:1980:CONVERGENCE}.

The special case $p = 1$ has its most important application in quantum physics. Characterization of the ground state of bosons is usually done with the Gross-Pitaevskii equation \cite{Jarlebring:2014:Inverse}, \cite{Altman:2019:JMETHOD}, whose spatial discretization is of the form \eqref{eq:main-p1}. Although SCF can be used in this case too, the more common techniques involve discretization of a gradient flow. See \cite{Bao:2004:BOSEEINSTEIN}, and references therein.

Another class of applications where $p = 1$ arises is in data science, for example, applications such as spectral clustering which rely on computing eigenpairs of the p-Laplacian \cite{Hein:2009:PLAPLACIAN}, \cite{Tudisco:2019:MATPOWER}, \cite{Hein:2010:IPM}. See \cite{DingLu:2018:ROBUST} for a Rayleigh quotient minimization approach for Fisher linear discriminant analysis, which is used in pattern recognition and classification. In \cite{Tudisco:2019:CORE}, the authors propose a new model for the core-periphery detection problem in network science (in the sense of \cite{Borgatti:2000:MODELS}) and show its equivalence to the $p = 1$ problem.

The contributions of the paper can be summarized as follows. In Section~\ref{sec:inv}, we introduce the concept of basis invariance. This allows us to derive an alternate characterization of \eqref{eq:main} in terms of an associated Jacobian. We introduce our implicit algorithms in Section~\ref{sec:impli} motivated by this result. Explicit procedures to carry out these algorithms are derived and studied in Section~\ref{sec:expli}. In Section~\ref{sec:convtheo}, we provide convergence results for these algorithms and Section~\ref{sec:simul} contains numerical examples, illustrating advantages of our approach.

We will extensively use vectorization and devectorization and
introduce the following shorthand. Small letters denote the vectorization of capital letters. For example,
\[
 x=\vect(X),\quad v^{(k)} = \vect(V^{(k)}),\quad s^{(k)} = \vect(S^{(k)}),\quad \vect(I_p) = i_p.
 \]
For any $H:\mathbb{R}^n\to\mathbb{R}^n$, the operator $\frac{dH}{dv}$ denotes forming the Jacobian of $H$ with respect to $v$, where $v$ denotes vectors in $\mathbb{R}^n$.

\section{Notion of invariant subspace}\label{sec:inv}
In order to appropriately generalize the concept
of invariant pairs, we will throughout the
paper make the following assumption on $A$.
\begin{assumption}[\textbf{Basis invariance}]\label{asmp:scalinv}
We consider $A:\mathbb{R}^{n\times p}\to M$ such that
it is a function of the outer product of $W$, i.e., 
\begin{equation}\label{eq:asmp1}
 A(W)=B(WW^T)
\end{equation}
for some $B:M\rightarrow M$. Moreover, we assume that 
\begin{equation}\label{eq:asmp2}
  B(X)=B(h(X))
\end{equation}
for any $X\in M$, where $h:\mathbb{R}\to \mathbb{R}$
denotes the heavyside function, generalized to a matrix sense,
\begin{equation}
h(t) = \begin{cases}1,\quad\quad\textrm{if}\;x> 0\\
                    0,\quad\quad\textrm{if}\;x\leq 0
       \end{cases}
\end{equation}
\end{assumption}
Assumption~\ref{asmp:scalinv} is a generalization of the scaling invariance property for the case $p = 1$ in
\cite{Jarlebring:2014:Inverse}. If $p=1$ and $v\in\mathbb{R}^{n}$, then for any $\alpha \in\mathbb{R}$,
\[
 A(\alpha v)=B(\alpha^2vv^T)=B(h(\alpha^2vv^T))=B(h(vv^T))=A(v).
\]
Moreover, Assumption~\ref{asmp:scalinv}
leads to the fact that $A(W)=A(WP)$ for invertible $P$,
as we shall illustrate in the following theorem.
This is important in our context, since it allows us
to interpret the columns of $W$ as a basis of an invariant
subspace, and $A$ can be viewed as a function of a subspace, i.e.,
it is a function of a vector space, and independent of the basis.
\begin{theorem}\label{thm:subspace}
  If $A$ satisfies the basis invariant 
  conditions  \eqref{eq:asmp1} and \eqref{eq:asmp2}
  then $A(W) = A(WP)$ for any non-singular matrix $P \in\mathbb{R}^{p\times p}$ and $W\in\mathbb{R}^{n\times p}$ where $n\le p$ and $\rank(W) = p$ .
\end{theorem}
  \begin{proof} 
    Let $W=QS$ for some invertible matrix $S$ and $Q\in\RR^{n\times p}$ orthogonal. Let $V,\Lambda_+$
    be a diagonalization of $SS^T$, i.e., $SS^T=V\Lambda_+V^T$, where $V$ is orthogonal
    and $\Lambda_+$ is a positive diagonal matrix.
 Then, 
\[
h(WW^T) = h(QSS^TQ^T) = h(QV\Lambda_{+}V^TQ^T) = QVh(\Lambda_+)V^TQ^T = QQ^T.
\]
This along with \eqref{eq:asmp1} and \eqref{eq:asmp2} gives us
\[
A(W) = B\left(h(WW^T)\right) = B(QQ^T) = A(Q).
\]
If we let $W=QR$ be a QR-factorization of $W$, we see that $A(W)=A(Q)$ with $S=R$,
and $A(WP)=A(Q)$ with $S=RP$.  This shows that $A(W)=A(WP)$, which
concludes the proof.
    \end{proof}

%
Since $S$ is symmetric, it can be diagonalized as $S = Q_s\Lambda_sQ_s^T$ where $Q_s$ is orthogonal. Problem \eqref{eq:main} can be reformulated using Theorem~\ref{thm:subspace} as
\begin{equation}\label{eq:probreform}
A(V)V = VQ_s\Lambda Q_s^T \implies A(VQ_s)VQ_s = VQ_s\Lambda_s.
\end{equation}
showing that a solution to \eqref{eq:main} can be diagonalized.
\begin{example}[Transformation to basis invariant form]
  The heaviside function usually does not appear directly
  in the standard formulation in NEPv-applications, but can be
  obtained easily.
  In the context of the self-consistent field
  iteration in quantum chemistry, we want to solve the equation
  \begin{equation}  \label{eq:HV}
  H(V)V=VS
  \end{equation}
  where, e.g. $H(V)=H_0+\diag(VV^T)$, which does not satisfy \eqref{eq:asmp1} and \eqref{eq:asmp2}.
  This can be transformed to a problem satisfying
  \eqref{eq:asmp1} and \eqref{eq:asmp2}
  by defining
  \begin{equation}\label{eq:AV_exmp}
  A(V):=H_0+\diag(h(VV^T)).
  \end{equation}
A pair $(V,S)$ is a full rank solution to \eqref{eq:HV} if and
  only if it is a solution to \eqref{eq:main} with $A$ defined as in \eqref{eq:AV_exmp}. However, the
  similarity transformation of a
  solution, i.e., $(VP,P^{-1}SP)$ is
  a solution to \eqref{eq:main}, but not \eqref{eq:HV}.
  \end{example}

We will denote the Jacobian as follows, and
we directly characterize a theoretical property
as a consequence of Assumption~\ref{asmp:scalinv}.
\begin{definition}[Left-hand side Jacobian]
The Jacobian of the vectorization of the LHS of the first subequation of \eqref{eq:main} is denoted as $J:\mathbb{R}^{np}\to\mathbb{R}^{np\times np}$ and given by
\begin{equation}\label{eq:jaclhs}
J(v) = I_p\otimes A(V)+\left(\frac{d}{dv}(I_p\otimes A(V))\hat{v}\right)_{\hat{v}=v}
\end{equation}
\end{definition}

The vectorized form of \eqref{eq:main} can now be written
 as
\begin{equation}\label{eq:main2}
F(V,S) := \begin{pmatrix}(I_p\otimes A(V))v-(I_p\otimes V)s \\
(I_p\otimes V^T)v-i_p \end{pmatrix} = \begin{pmatrix}(I_p\otimes A(V))v-(S^T\otimes I_n)v \\
(I_p\otimes V^T)v-i_p \end{pmatrix} = 0.
\end{equation}
The method we propose will work better for problems
where the Jacobian evaluated in the solution is non-singular.
The Jacobian of \eqref{eq:main2} in the fixed point is given by
\begin{equation}  \label{eq:Jacobian}
\begin{bmatrix}
J(v_*)-{S_*}^T\otimes I_n&-I_p\otimes V_*\\
Z_*&0\\
\end{bmatrix}
\end{equation}
where 
\begin{equation*}
Z_*=\left(\frac{d}{dv}\vect(V^TV)\right)_{V=V^{*}}.
\end{equation*}

As a conseqeunce of the Assumption~\ref{asmp:scalinv},
we conclude the following generalization of \cite[Lemma~2.1]{Jarlebring:2014:Inverse}, which shows a relationship between the eigenpairs of $J$ and $A$. We exploit this relationship later in Section~\ref{sec:impli} and Section~\ref{sec:expli} where we formulate and derive our algorithms respectively.
\begin{theorem}[Eigenproblem equivalence]
\label{thm:eigeqv}
  For any $v\in\RR^{np}$, we have
  \[
J(v)v=(I_p\otimes A(V))v.
  \]
\end{theorem}
  \begin{proof}
  From \eqref{eq:jaclhs}, we have
  \[
  \Bigg(J(v)-(I_p \otimes A(V))\Bigg)v = \left(\frac{d}{dv}(I_p\otimes A(V))\hat{v}\right)_{\hat{v}=v}v.
  \]
  Interpreting the above as a directional derivative in the direction of $v$, we have
  \begin{equation*}
\begin{split}
 \left(\frac{d}{dv}(I_p\otimes A(V))\hat{v}\right)_{\hat{v}=v}v &= \lim\limits_{\epsilon \to 0} \frac{\Bigg(I_p\otimes B\Big((V+\epsilon V)(V+\epsilon V)^H\Big)-I_p\otimes B\Big(VV^H\Big)\Bigg)v}{\epsilon}\\
    &= \lim\limits_{\epsilon \to 0} \frac{\Bigg(I_p\otimes B\Big(h\left((\epsilon+1)^2VV^H\right)\Big)-I_p\otimes B\Big(VV^H\Big)\Bigg)v}{\epsilon}\\
    &= \lim\limits_{\epsilon \to 0} \frac{\Bigg(I_p\otimes B\Big(VV^H\Big)-I_p\otimes B\Big(VV^H\Big)\Bigg)v}{\epsilon}\\
    &= 0.
\end{split}
\end{equation*}
This completes the proof.
  \end{proof}



\section{Implicit algorithms}\label{sec:impli}
Standard Newton's method for the vectorized form \eqref{eq:main2} is
\begin{equation}\label{eq:newt}
-F^{(k)}=\begin{bmatrix}
J(v^{(k)})-{S^{(k)}}^T\otimes I_n&-I_p\otimes V^{(k)}\\
Z_k&0\\
\end{bmatrix}\begin{bmatrix}
  \Delta v^{(k)}\\
  \Delta s^{(k)}
\end{bmatrix}
\end{equation}
where the $\Delta$-matrices are updates,
$v^{(k+1)}=v^{(k)}+\Delta v^{(k)}$ and
$s^{(k+1)}=s^{(k)}+\Delta s^{(k)}$, with
\begin{equation}\label{eq:yk}
Z_k=\left(\frac{d}{dv}\vect(V^TV)\right)_{V=V^{(k)}}.
\end{equation}
where 
\begin{equation*}
F^{(k)} := F(V^{(k)},S^{(k)}).
\end{equation*}

The iterates $V^{(k+1)}$ computed using \eqref{eq:newt} and \eqref{eq:yk} are not orthogonal matrices. We prefer iterates which are orthogonal
since the solution is orthogonal and it is advisable to work with orthogonal iterates from a robustness perspective. We will now carry out a modification
such that
\begin{equation}\label{eq:norm_iterates}
  {V^{(k)}}^TV^{(k)} = I_p.
\end{equation}
In this first modification of Newton's method,  we use an alternate definition of $Z_k$
\begin{equation}\label{eq:zalt}
Z_k=\frac12\left(\left(\frac{d}{dv}\vect(V^TV)\right)_{V=V^{(k)}}+
\left(\frac{d}{dv}\vect(V^TV)\right)_{V=V^{(k+1)}}\right).
\end{equation}
The reason this implies \eqref{eq:norm_iterates} can be seen as follows.
We first observe the relation
from the product rule 
\begin{equation}\label{eq:WVT_VWT}
 \left(\frac{d}{dv}\vect(V^TV)\right)w=\vect(W^TV+V^TW).
 \end{equation}
 for an arbitrary vector $w$. 
 The last block equation in the update \eqref{eq:newt} now
 reads 
 \begin{equation}  \label{eq:newt_last_eq}   
 Z_k\vect(v^{(k+1)}-v^{(k)})=-(I_p\otimes (V^{(k)})^T)v^{(k)}+i_p
 \end{equation}

 Reversing the vectorization, and using \eqref{eq:WVT_VWT}
 the left hand side can be expanded and simplified to 
\begin{multline*}
 \frac12 \left(
 (V^{(k+1)}-V^{(k)})^TV^{(k)}+(V^{(k)})^T(V^{(k+1)}-V^{(k)})+\right.\\
 \left.
 (V^{(k+1)}-V^{(k)})^TV^{(k+1)}+(V^{(k+1)})^T(V^{(k+1)}-V^{(k)})
 \right) =\\
 (V^{(k+1)})^TV^{(k+1)}-(V^{(k)})^TV^{(k)}.
\end{multline*}
The insertion into \eqref{eq:newt_last_eq} leads to ${V^{(k+1)}}^TV^{(k+1)}=I_p$,
i.e., \eqref{eq:norm_iterates} is satisfied for $k=2,3,\ldots$.

We introduce this first implicit method
in order to maintain orthogonality of the eigenvector
approximations. This is summarized in the
following lemma, which also illustrates
that the behavior is not that different
from the standard Newton method. The proof of statement (b)
is based on theory for Newton-like methods and
can be found in Appendix~\ref{sec:convproof}.

\begin{lemma} \label{thm:newt_mod1}
  Let $v^{(k)}$, $s^{(k)}$, $k=1,\ldots$, be a sequence of vectors that
  satisfy \eqref{eq:newt} with $Z_k$ given by \eqref{eq:zalt}.
  Then,
  \begin{itemize}
  \item[(a)] $(V^{(k)})^TV^{(k)}=I_p$ for $k=2,3,\ldots$
  \item[(b)] If the sequence converges monotonically to a solution $(V_*,S_*)$, and the Jacobian evaluated in $(V_*,S_*)$ (given by \eqref{eq:Jacobian}) is invertible, then it converges with the same convergence order as Newton's method.
  \end{itemize}
\end{lemma}


In this work we consider two modifications of the Newton-like method
in Lemma~\ref{thm:newt_mod1}. Both lead either to new methods (which
have some attractive properties) or well-established methods
suggesting that the methods can be viewed as Newton-like methods.

\begin{itemize}
  \item 
We modify the (1,2) block of the Jacobian to $I_p\otimes V^{(k+1)}$.
\begin{equation}\label{eq:qn1}
-F^{(k)}=\begin{bmatrix}
J(v^{(k)})-{S^{(k)}}^T\otimes I_n&-I_p\otimes V^{(k+1)}\\
Z_k&0\\
\end{bmatrix}\begin{bmatrix}
  \Delta v^{(k)}\\
  \Delta s^{(k)}
\end{bmatrix}
\end{equation}
This is analogous to the modification \cite[Equation (1.10)]{Jarlebring:2017:QUASINEWTON} which directly leads to the method of successive linear problems \cite{Ruhe:1973:NLEVP}. With the techniques of the next section, this leads to Algorithm~\ref{alg:jversion}.
\item We modify the (1,1) block of the Newton's method Jacobian from $J(v^{(k)})$ to $I_p\otimes A(V^{k})$, in addition to the modification done to obtain Algorithm~\ref{alg:jversion}.
\begin{equation}\label{eq:qn2}
-F^{(k)}=\begin{bmatrix}
I_p\otimes A(V^{(k)})-{S^{(k)}}^T\otimes I_n&-I_p\otimes V^{(k+1)}\\
Z_k&0\\
\end{bmatrix}\begin{bmatrix}
  \Delta v^{(k)}\\
  \Delta s^{(k)}
\end{bmatrix}
\end{equation}
 We will show that this leads to Algorithm~\ref{alg:aversion}, which is the well known SCF iteration.
\end{itemize}

\section{Explicit interpretation of algorithms}\label{sec:expli}
%
%
%
%
%
Both update equations \eqref{eq:qn1} and \eqref{eq:qn2} correspond to implicit methods.
We will illustrate several situations where we can generate iterates that satisfy
the implicit algorithms update equations in an explicit way.

Equations \eqref{eq:qn1} and \eqref{eq:qn2} have to be expanded to derive usable algorithms, starting
with the latter variant since it leads to a clear relationship with state-of-the-art methods.

\subsection{Algorithm \ref{alg:aversion}}
Although, Algorithm~\ref{alg:aversion} is a modification
of Algorithm~\ref{alg:jversion}, we start our discussion
with Algorithm~\ref{alg:aversion} since it leads to a well-established method.
We can obtain Algorithm~\ref{alg:aversion} can be obtained from \eqref{eq:qn2} by multiplying out the first subequation of \eqref{eq:qn2} as follows.
\begin{equation*}
\begin{split}
&\Big((I_p\otimes A(V^{(k)}))-({S^{(k)}}^T\otimes I_n)\Big)(v^{(k+1)}-v^{(k)})-(I_p\otimes V^{(k+1)})(s^{(k+1)}-s^{(k)}) \\
&=\\ & -\Big((I_p\otimes A(V^{(k)}))-({S^{(k)}}^T\otimes I_n)\Big)v^{(k)}
\end{split}
\end{equation*}
Cancellation of terms leads to
\begin{equation}\label{eq:scf_non_vectorized}
\Big(I_p\otimes A(V^{(k)})\Big)v^{(k+1)} = (S^{(k+1)}\otimes I_n)v^{(k+1)}.
\end{equation}
Devectorizing this system gives the following result.

\begin{theorem}\label{thm:qn2appl}
  Let $(V^{(k)},S^{(k)})$ be pairs that satisfy the
  update equation \eqref{eq:qn2} for $k=1,\ldots$. Then, 
\begin{equation}\label{eq:scf}
 A(V^{(k)})V^{(k+1)} = V^{(k+1)}S^{(k+1)},
\end{equation}
where $(V^{(k)})^TV^{(k)}=I_p$ for $k=2,\ldots$.
\end{theorem}

This leads to a practical way to carry out the algorithm, since \eqref{eq:scf} is
an eigenvalue problem where $V^{(k+1)}$ are the eigenvectors and the diagonal matrix $S^{(k+1)}$
are the eigenvalues. This is the well-known SCF algorithm.

%

\subsection{Algorithm~\ref{alg:jversion}}
Similarly, Algorithm~\ref{alg:jversion} can be obtained from the first subequation in \eqref{eq:qn1} as follows.
\begin{equation}\label{eq:qn11}
\begin{split}
&
\bigg(J(v^{k})-({S^{(k)}}^T\otimes I_n)\bigg)({v}^{(k+1)}-v^{(k)}) \\ &= (I_p\kron {V}^{(k+1)})(s^{(k+1)}-s^{(k)})+({S^{(k)}}^{T}\otimes I_n)v^{(k)}-(I_p\otimes A(V^{(k)}))v^{(k)}\\
&= \bigg(\left({S^{(k+1)}}^T-{s^{(k)}}^T\right)\otimes I_n\bigg)v^{(k+1)}+({S^{(k)}}^{T}\otimes I_n)v^{(k)}-(I_p\otimes A(V^{(k)}))v^{(k)}.
\end{split}
\end{equation}
From Theorem~\ref{thm:eigeqv}, we have $J(v^{(k)})v^{(k)} = \left(I_p\otimes A(V^{k})\right)v^{(k)}$. Using this in \eqref{eq:qn11} leads to the following result.
\begin{theorem}\label{thm:qn1appl}
  Let $(V^{(k)},S^{(k)})$ be pairs that satisfy the
  update equation \eqref{eq:qn1} for $k=1,\ldots$. Then, 
\begin{equation}\label{eq:qn1finale}
J(v^{(k)}){v}^{k+1} = ({S^{(k+1)}}^T\kron I_n)v^{(k+1)}.
\end{equation}
\end{theorem}
Since $J$ is not block diagonal, \eqref{eq:qn1finale} cannot be easily devectorized as was done for \eqref{eq:scf_non_vectorized}.
For the special case $p=1$ we directly
identify that \eqref{eq:qn1finale} reduces to 
\begin{equation}\label{eq:jscf}
J(v^{(k)})v^{(k+1)} = \lambda^{(k+1)} v^{(k+1)}.
\end{equation}
Similar to \eqref{eq:scf}, \eqref{eq:jscf} is a standard eigenvalue problem
and we can compute a next iterate with a solver for standard eigenvalue problem.
It directly suggests that the matrix $A(V^{(k)})$ in the SCF-iteration, can be viewed
as an approximation of the Jacobian matrix, and in order to obtain
faster convergence it can be better to use $J(v^{(k)})$, or approximations
thereof.
This in turn leads to quadratic convergence and in contrast to
Newton's method, the method converges in one step for a linear problem,
and is superior to Newton's method for problems that are close to being linear in this sense (as we prove in Section~\ref{subsec:singlestep}).

\subsection{Implementation aspects}\label{sec:implidet}
In both the implicit algorithms, the first step in the loop can be done with standard tools for solving linear eigenproblems \cite{Bai:2000:TEMPLATES}. The iterates of both algorithms will depend on which $p$ eigenvectors are selected to construct $V^{(k+1)}$. The best choice is usually application dependent. For example, in applications based on quantum mechanics, one is usually interested in the $p$ first eigenvalues and hence, the eigenvectors corresponding to the $p$ first eigenvalues would be selected in each step. 

Although Theorem~\ref{thm:qn2appl} and Theorem~\ref{thm:qn1appl} provide us with explicit ways to implement our implicitly formulated methods, they do not automatically enforce the orthogonality constraint ${V^{(k+1)}}^T V^{(k+1)} = I_p$. To this end, we compute an intermediate eigenvector eigenpair $\left(Y,Z\right)$ and add an additional step in our algorithms that involves computing a thin QR factorization of $Z$ to obtain $V^{(k+1)}$. This improves the numerical stability of our implementation.

\begin{algorithm}[h]
\SetAlgoLined
\SetKwInOut{Input}{input}
\SetKwInOut{Output}{output}
\Input{$V^{(0)} \in \mathbb{R}^{n\times p}$ such that ${V^{(0)}}^TV^{(0)} = I$, Tolerance TOL}
\Output{$(V_*, S_*)\in\mathbb{R}^{n\times p}\times\mathbb{R}^{p\times p}$ that satisfies \eqref{eq:main}}
 
 \For{$k = 0,1,\ldots,$ until convergence}{
   Find $\left(Y,Z\right)$ such that \[J(v^{(k)})y = (Z^T\kron I_n)y.\]\\
   Compute thin QR factorizaton of $Y$ to get $V^{(k+1)}$. \[Y = V^{(k+1)}R.\]\\
  Compute a similarity transformation of Z: $S^{(k+1)} = R^{-1}ZR.$
 }
\caption{J-version \label{alg:jversion}}
\end{algorithm}

\begin{algorithm}[h]
\SetAlgoLined
\SetKwInOut{Input}{input}
\SetKwInOut{Output}{output}
\Input{$V^{(0)} \in \mathbb{R}^{n\times p}$ such that ${V^{(0)}}^TV^{(0)} = I$, Tolerance TOL}
\Output{$(V_*, S_*)\in\mathbb{R}^{n\times p}\times\mathbb{R}^{p\times p}$ that satisfies \eqref{eq:main}}
 
 \For{$k = 0,1,\ldots,$ until convergence}{
   Find $\left(Y,Z\right)$ such that \[A(V^{(k)})Y = YZ.\]\\
   Compute thin QR factorizaton of $Y$ to get $V^{(k+1)}$. \[Y = V^{(k+1)}R.\]\\
   Compute a similarity transformation of Z: $S^{(k+1)} = R^{-1}ZR.$\\
 }
 \caption{A-version \label{alg:aversion}}
\end{algorithm}

Note that the final step in both algorithms need not be done for every iteration of the loop since we do not use $S^{(k+1)}$ anywhere inside it. We can instead compute the transformation only once after the termination of the loop.

We do not provide an explicit procedure for \eqref{eq:qn1finale} for
$p>1$, making the procedure somewhat theoretical for that case.
There are very important applications for the $p=1$ case, and we
illustrate in the simulations that if we can solve \eqref{eq:qn1finale} corresponding to $p>1$ in some way (possibly approximately)
we obtain attractive convergence properties.
\section{Convergence theory}\label{sec:convtheo}
\subsection{Local convergence of Algorithm 2}
Since Algorithm~\ref{alg:aversion} is equivalent to SCF as shown in Theorem~\ref{thm:qn2appl}, the convergence can be described in the setting of SCF.
There has been extensive study of convergence of SCF and its acceleration in the last fifty years. Several results exist in the literature, as mentioned in Section~\ref{sec:intro}. In general, SCF exhibits linear local convergence when it converges.
Convergence can be characterized in terms of gaps \cite{Yang:2009:SCF}
(see also \cite[Theorem~3.1]{ParikElias:2018:DENSITY}). Recent gloval
convergence are available, e.g., in \cite{liu2015analysis}, \cite{liu2014convergence}.
\subsection{Local convergence of Algorithm 1}
Due to our inexact Newton viewpoint, the convergence of Algorithm~\ref{alg:jversion} can be characterized using results in the rich literature on inexact Newton methods. Quadratic local convergence can be proved using theorems in \cite{Dembo:1982:INEXACT}.
\begin{theorem} \label{thm:qn1_conv}
  Let $(V^{(k)}$, $S^{(k)})$, $k=1,\ldots$, be a sequence of pairs that satisfies 
  satisfy \eqref{eq:qn1finale} with the normalization constraint ${V^{(k)}}^T V^{(k)} = I_p$.
  If the sequence converges monotonically to a solution $(V_*,S_*)$, and the Jacobian evaluated in $(V_*,S_*)$ (given by \eqref{eq:Jacobian}) is invertible, then it converges with the same convergence order as Newton's method.
\end{theorem}
We refer the reader to Appendix~\ref{sec:convproof} for the proof. 
\subsection{Single step analysis}\label {subsec:singlestep}
As we illustrate in the examples, the implicit methods
tend seem to often work considerably better in general
and in particular for close-to-linear problems. This is intuitively
natural since both implicit methods converge
in one step if we apply it to linear problems.

This can be further characterized, by considering one
step of the method applied to a problem parameterized
by a parameter $\alpha$, where $\alpha=0$ corresponds to
a linear problem. For this analysis we consider the model problem
\begin{equation*}
 A(v)=A_0+\alpha C(v).
\end{equation*}

Let $\begin{bmatrix}v_0,&\lambda_0\end{bmatrix}^T$ be an initial guess and $\begin{bmatrix}v_+,&\lambda_+\end{bmatrix}^T$ be the result of (for the moment)
any of the two algorithms.
We introduce three functions
\[
G_\beta\left(\begin{bmatrix}
\lambda \\ v\\\alpha
\end{bmatrix}\right)=\begin{bmatrix}
(A_0+\alpha P_\beta)v-\lambda v\\
v^Tv-1
\end{bmatrix}
\]
where $\beta$ can be  $\beta=*$, $\beta=A$ and $\beta=J$.
The values of $P$ (where we dropped the parameters
for notational convenience) denote the nonlinearity
\begin{subequations}
\begin{eqnarray}
  P_*&=& C(v_*)\\
  P_A&=& C(v_0)\\
  P_J&=& \left(\frac{d}{dv}\left(C(v)v\right)\right)_{v=v_0}.
\end{eqnarray}
\end{subequations}
These nonlinear functions respectively corresponds the residual
for
the exact solution ($\beta=*$), one step of Algorithm~\ref{alg:jversion} ($\beta=J$) and one step of Algorithm~\ref{alg:aversion} ($\beta=A$). 
%

We can apply the implicit function
theorem for all three functions, and express
the first $n+1$ variables in terms of the third
variable $\alpha$ in a neighborhood of the solution,
if the associated Jacobian is non-singular.
The Jacobian given \eqref{eq:Jacobian}
is now assumed to be non-singular in the solution. 
The exact solution can then be expanded as
\begin{equation}\label{eq:single1}
\begin{bmatrix}
  v_*(\alpha)\\
  \lambda_*(\alpha)
\end{bmatrix}=
\begin{bmatrix}
  v_*(0)\\
  \lambda_*(0)
\end{bmatrix}
-\alpha \begin{bmatrix}
 A_0-\lambda_*(0)I &-v  \\
 2v^T &     \\
\end{bmatrix}^{-1}\begin{bmatrix}
  C(v)v\\
  0
\end{bmatrix}
+\mathcal{O}(\alpha^2)
\end{equation}
whereas both $\beta=A$ and $\beta=J$ can be expanded
as 
\begin{equation}\label{eq:single2}
\begin{bmatrix}
  v_+(\alpha)\\
  \lambda_+(\alpha)
\end{bmatrix}=
\begin{bmatrix}
  v_+(0)\\
  \lambda_+(0)
\end{bmatrix}
-\alpha \begin{bmatrix}
 A_0-\lambda_*(0)I &-v  \\
 2v^T &     \\
\end{bmatrix}^{-1}\begin{bmatrix}
  c_\beta\\
  0
\end{bmatrix}
+\mathcal{O}(\alpha^2)
\end{equation}
where $c_A=P_av_*(0)$ and $c_J=P_jv_*(0)$.

The first terms in the Taylor expansion of
the next iterate and the exact iterate
as a function of the parameterization of the
nonlinearity, are equal.
Therefore,
\[
 \begin{bmatrix}
  v_+(\alpha)\\
  \lambda_+(\alpha)
\end{bmatrix}-\begin{bmatrix}
  v_*(\alpha)\\
  \lambda_*(\alpha)
\end{bmatrix}=\mathcal{O}(\alpha)
\]
meaning that the accuracy of one step, is order
of magnitude of the nonlinear term. Moreover,
the coefficient is proportional to
$\|C(v_*(0))v_*(0)-P_\beta v_*(0)\|$.

\section{Simulations}\label{sec:simul}
\subsection{Scalar nonlinearity}
The theory and methods are first illustrated with a reproducible example where $p = 1$. We consider \eqref{eq:main-p1} with
\begin{equation}\label{eq:exp1}
A(v) = A_0+\alpha\sin\left(\frac{v^TA_2v}{v^Tv}\right)A_1
\end{equation}
and
\begin{eqnarray*}
A_0 &=& \frac{1}{10}\begin{pmatrix}10& 21& 13& 16\\21& -26& 24& 2\\13& 24& -26& 37\\16& 2& 37& -4\end{pmatrix},\quad A_1 = \frac{1}{10}\begin{pmatrix}20& 28& 12& 32\\ 28& 4& 14& 6\\12& 14& 32& 34\\ 32& 6& 34& 16 \end{pmatrix},\\
A_2 &=& \frac{1}{10}\begin{pmatrix}-14& 16& -4& 15\\ 16& 10& 15& -9\\ -4& 15& 16& 6\\15& -9& 6&-6 \end{pmatrix}
\end{eqnarray*}
and $\alpha \in \mathbb{R}$. Note that $A$ in \eqref{eq:exp1} satisfies Assumption~\ref{asmp:scalinv} if we select 
\[
B(X) = A_0+\alpha \sin\left(\frac{c^TXBXc}{c^TXc}\right)A_1
\]
for essentially any $c\in\mathbb{R}^n$. This specific example appears in \cite[Section~3.3]{Jarlebring:2014:Inverse} and $J$ is explicitly given by
\begin{equation*}
J(v) = A(v)+2\alpha\frac{\cos(\frac{v^TA_2v}{v^Tv})}{{(v^Tv)}^2}A_1v((v^Tv)v^TA_2-(v^TA_2v)v^T).
\end{equation*}

We solve four instances of this problem generated by four different values of $\alpha$, that is $\alpha = 0,0.5,1$ and $5$. To all of these instances, we apply Algorithm~1 (using \eqref{eq:jscf}), Algorithm~2, the J-Inverse iteration (from \cite{Jarlebring:2014:Inverse}) and Newton's method with initial guess $v_0 = \begin{pmatrix}1,&1,&1,&1\end{pmatrix}^T$. In Figure~\ref{fig:ex1}, we see the error history of all three methods for all four values of $\alpha$. The error is computed as $\|v_k-v_*\|$, where $v_*$ is the reference solution.

 We observe linear convergence for Algorithm 2 and quadratic convergence for Algorithm 1 as predicted by the theory in Section~\ref{sec:convtheo}. Both implicit methods are competitive, at least for small values of $\alpha$. For higher values of $\alpha$, the number of iterations required to enter the regime of quadratic convergene increases for both Algorithm 1 and Newton's method. This example illustrates a simple case when Algorithm 1 is a better choice than Newton's method, although both methods converge quadratically. We observe linear convergence for J-Inverse iteration as predicted by \cite[Theorem~3.1]{Jarlebring:2014:Inverse}.

 In Figure~\ref{fig:onestep}, we visualize the implications of the theory in Section~\ref{subsec:singlestep} by plotting the single step errors for all three methods. It is clear that the single step error is linear in $\alpha$, as expected from \eqref{eq:single1} and \eqref{eq:single2}. The predicted line is plotted using the coefficent $\|C(v_*(0))v_*(0)-P_\beta v_*(0)\|$. This illustrates an advantage of the proposed methods for small $\alpha$. 

\begin{figure}[htbp!]
     \subfigure{\label{fig:ex11}\includegraphics[width=60mm]{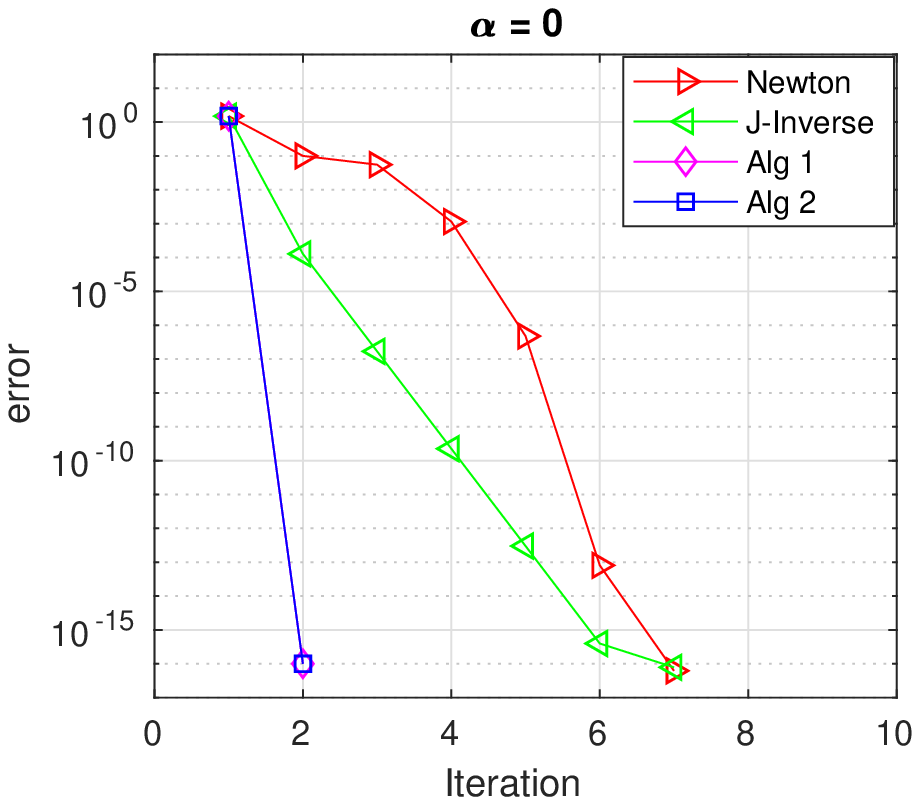}}
    \subfigure{\label{fig:ex12}\includegraphics[width=60mm]{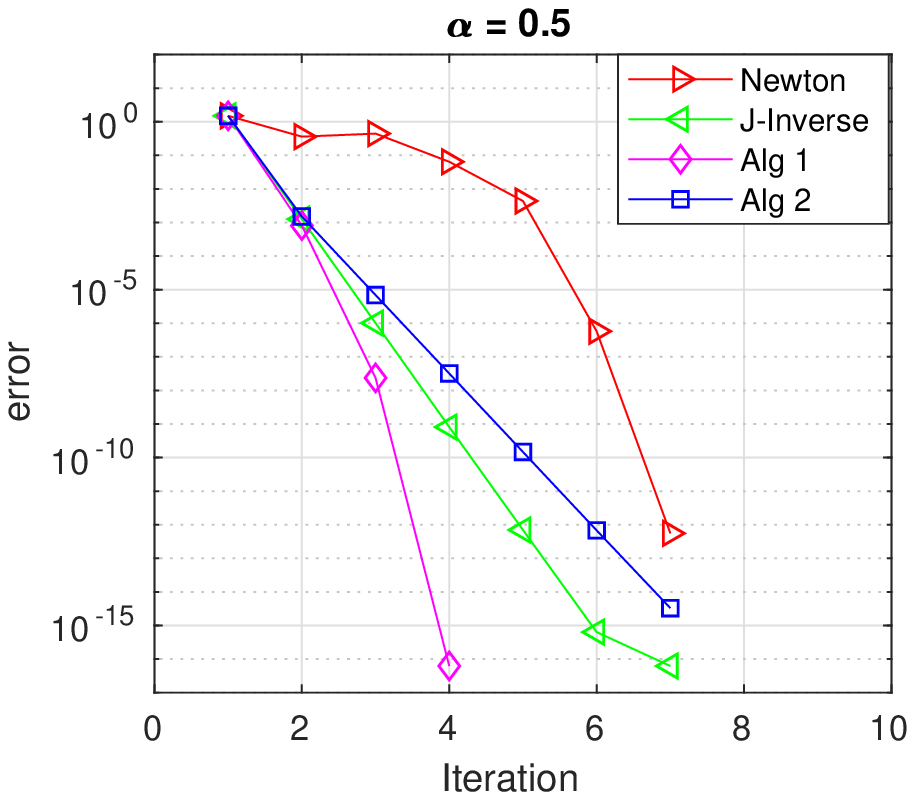}}
     \subfigure{\label{fig:ex13}\includegraphics[width=60mm]{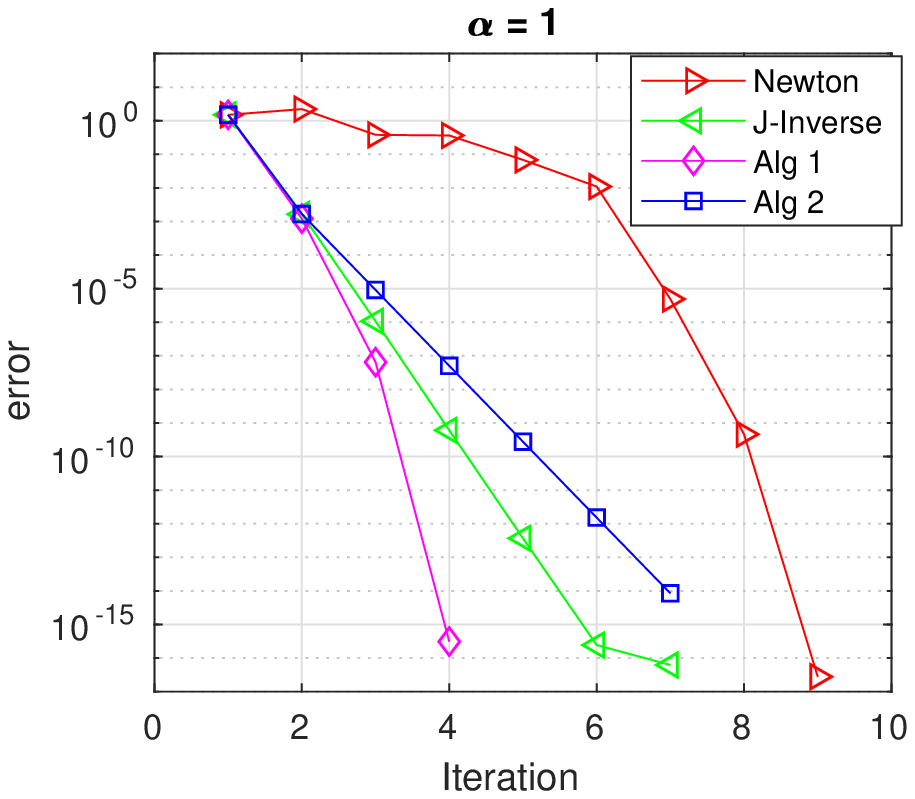}}
     \subfigure{\label{fig:ex14}\includegraphics[width=60mm]{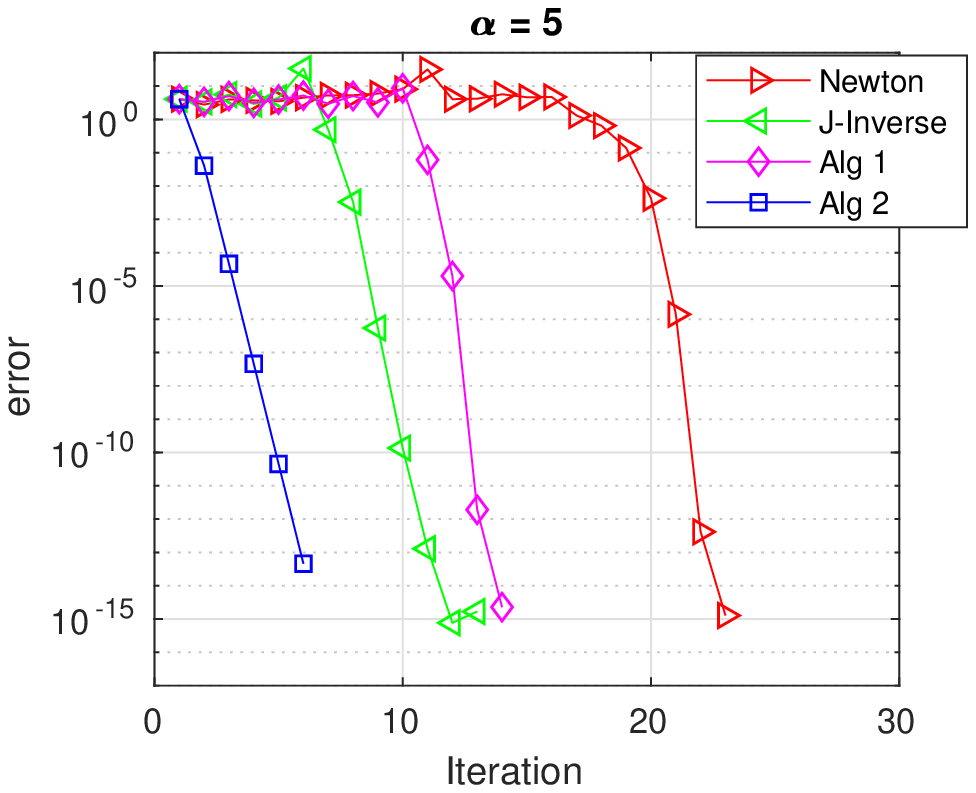}}
      
    \caption{Algortihm~1, Algorithm~2, J-Inverse iteration and Newton's method for different $\alpha$}
    \label{fig:ex1}
\end{figure}

\begin{figure}[h]
  \begin{center}
    \scalebox{0.7}{\includegraphics{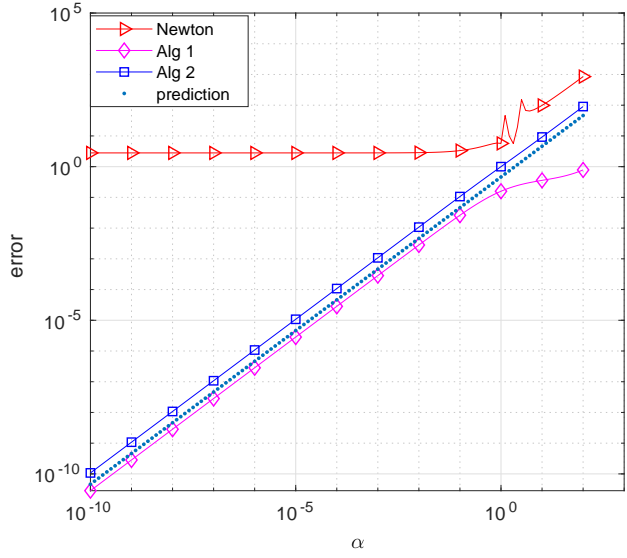}}
    \caption{Dependence of single step error on $\alpha$}
      \label{fig:onestep}

  \end{center}
\end{figure}

\subsection{Computing the ground state of bosons}\label{subsec:gpe}
The Gross-Pitaevskii equation (GPE) is a nonlinear PDE obtained by a Hartree-Fock approximation (see \cite{Saad:2010:ELECSTRUCT}) of the Schrödinger equation. It describes the ground state of identical bosons in a quantum system. We consider the case of a rotating Bose-Einstein condensate on the domain $\mathbb{D}=(-L,L)\times (-L,L)$. In this case, the GPE for the wave function $\Psi: \mathbb{R}^2\to\mathbb{C}$ under an external potential $V:\mathbb{R}^2\to\mathbb{R}$ is
\begin{equation}\label{eq:gpe}
\left(-\frac{1}{2}\Delta-i\Omega\dfrac{\partial}{\partial \phi}\Psi(x,y)+V(x,y)\right)+b|\Psi(x,y)|^2\Psi(x,y) = \lambda \Psi(x,y),\quad (x,y)\in\mathbb{D}.
\end{equation}  
Here, $\dfrac{\partial}{\partial \phi} = y\dfrac{\partial}{\partial x}-x\dfrac{\partial}{\partial y}$.
The scalar $b$ is a constant indicating the strength of interaction between the bosons and $\Omega$ is the angular velocity of rotation. We choose the boundary condition $\Psi(x,y) = 0$ for $(x,y)\in\partial D$. 

 We perform a central difference discretization of \eqref{eq:gpe} using a uniform grid of $N+2$ points along each dimension with grid spacing $\Delta x$. Details are in \cite[Section~5.1]{Jarlebring:2014:Inverse}. This leads to a problem of size $n = 2N^2$ with
\begin{eqnarray}\label{eq:gpeprob}
A(v) &=& \begin{pmatrix}\re \tilde{A}_0& -\im \tilde{A}_0\\ \im \tilde{A}_0&  \re \tilde{A}_0\end{pmatrix} +\frac{\gamma}{v^Tv}B(v),\nonumber\\
B(v) &=& \begin{pmatrix}\diag(v_1)+\diag(v_2)& 0\\0&  \diag(v_1)+\diag(v_2)\end{pmatrix}^2\nonumber\\
v &=& \begin{pmatrix}v_1,& v_2\end{pmatrix}^T. \nonumber
\end{eqnarray}
where $\tilde{A}_0$ is the discretization of the linear operator $-\frac{1}{2}\Delta-i\Omega\dfrac{\partial}{\partial \phi}+V(x,y)$ and $\gamma = b(\Delta x)^{-2}$. 
Note than both $v_1, v_2 \in\mathbb{R}^{N^2}$ and $v_1+iv_2$ gives the vectorization of $\Psi$ evaluated at the interior points. We have
\begin{equation}\label{eq:J_gpe}
\begin{split} 
J(v) =& \begin{pmatrix}\re \tilde{A}_0& -\im \tilde{A}_0\\ \im \tilde{A}_0&  \re \tilde{A}_0\end{pmatrix}\\ &+\frac{\alpha}{v^Tv}\left[\begin{pmatrix}3\diag(v_1)^2+\diag(v_2)^2& 2\diag(v_1)\diag(v_2)\\ 2\diag(v_1)\diag(v_2)& \diag(v_1)^2+3\diag(v_2)^2\end{pmatrix}-\frac{2}{v^Tv}B(v)vv^T\right].
\end{split}
\end{equation} 
\begin{figure}[htbp!]
     \subfigure{\label{fig:ex21}\includegraphics[width=60mm]{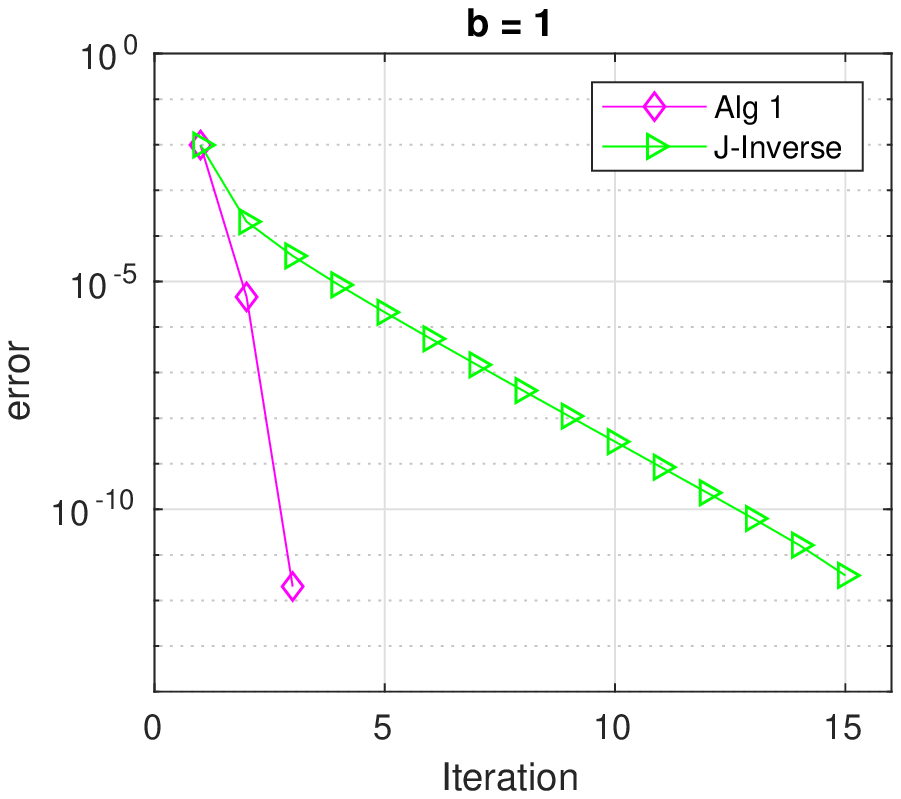}}
     \subfigure{\label{fig:ex22}\includegraphics[width=60mm]{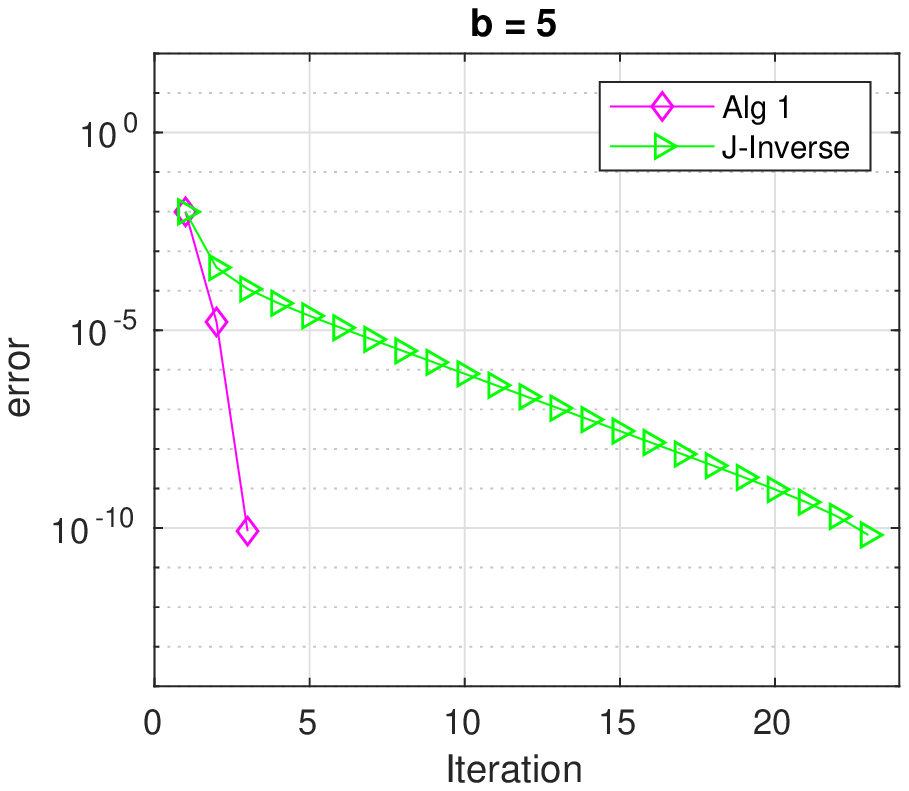}}
     \subfigure{\label{fig:ex23}\includegraphics[width=60mm]{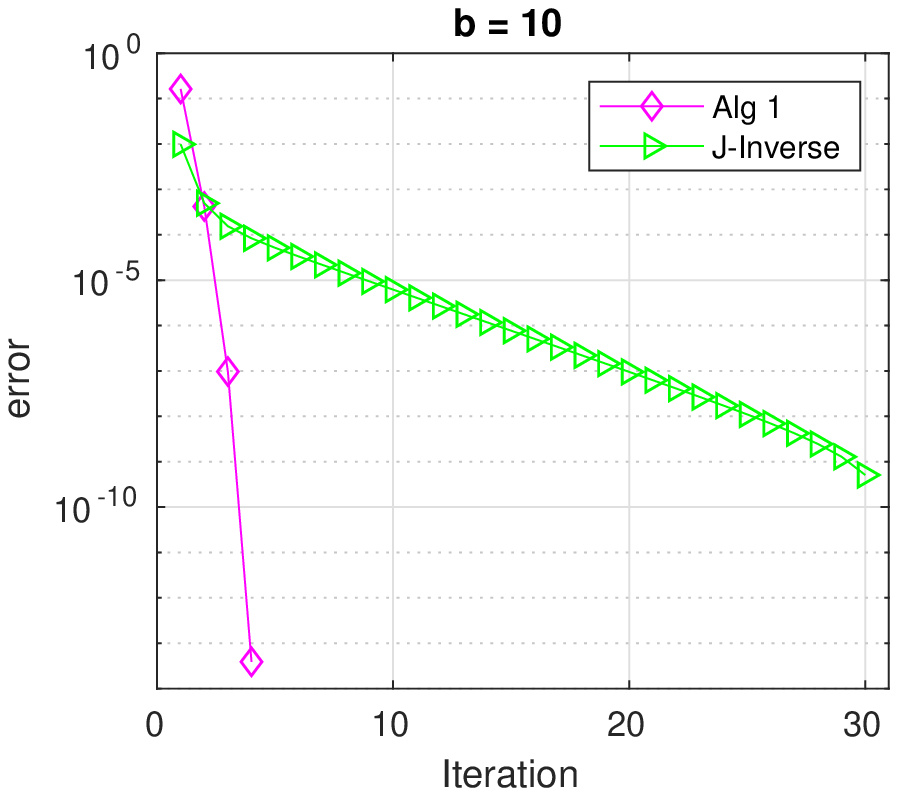}}
     \subfigure{\label{fig:ex24}\includegraphics[width=60mm]{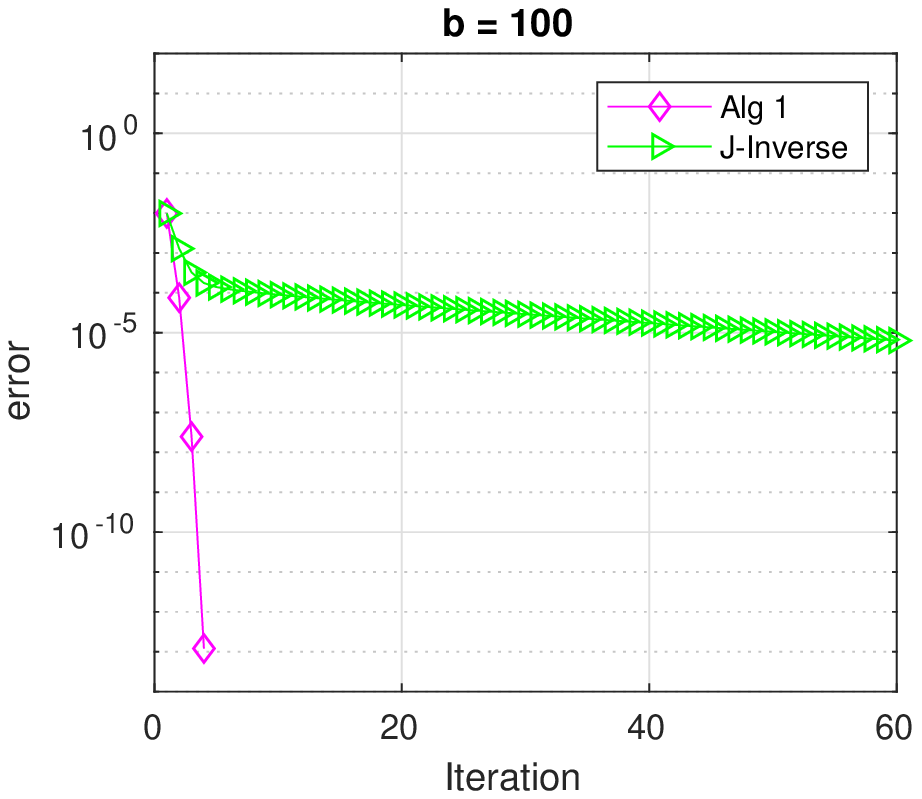}}

     \caption{Algorithm~\ref{alg:jversion} and J-Inverse iteration applied to GPE for different $b$}
    \label{fig:ex2}
\end{figure}
We use Algorithm 1 on with $J$ as defined by \eqref{eq:J_gpe}.

Since one step of both Algorithm~1 and Algorithm~2 requires
the solution of a standard eigenvalue problem, we need to
select an appropriate eigenpair at each step. Special attention
is needed in the selection in this problem. We select a new iterate
in a way that minimizes the difference between two iterates.
More precisely, we choose $\delta \in (0,1)$ and select all eigenpairs which correspond to eigenvalue within a radius $\delta$ of a given target.
We then do a least squares fitting to find the linear combination of these eigenvectors which is closest to the previous iterate.
This is needed due to the fact that the
problem has highly clustered eigenvalues.

\subsection{Invariant subspace}\label{subsec:invarsubs}
We consider
\begin{equation}\label{eq:yang}
A(V) = A_0+\alpha\diag({A_0}^{-1}\diag(h(VV^T)))
\end{equation}
where $A_0$ is the discrete 1D Laplacian. Problems of this type occur frequently in electronic structure calculations when using a Hartree-Fock discretization of the Schrödinger equation. See \cite{Yang:2009:SCF} for a discussion of the problem type and convergence results of SCF applied to \eqref{eq:yang}. 

If we let $e_i$ denote the the $i$th column of the identity matrix $I_n$ and $E_{i,j} = e_ie_j^T$, then
\[
A(V) = A_0+\alpha\sum\limits_{i=1}^{n}(e_i^Th(VV^T)e_i)\diag(A_0^{-1}E_{i,i}).
\]
We refer the reader to Appendix~\ref{sec:deriv} for a derivation of $J(V)$ for problems of this type.

The implementation of Algorithm~\ref{alg:aversion} is straightforward from \eqref{eq:scf}. We also illustrate the importance of the implicit formulation of Algorithm~\ref{alg:jversion} by inexactly solving \eqref{eq:qn1finale}. We do this with the optimization subroutine \textit{fminsearch}. It provides us a way to test Algorithm 1 for relatively small examples as a proof of concept. We use $n = 10$, $p = 3$ and apply Algorithm 1 and Algorithm 2 for two different values of $\alpha$.
\begin{figure}[htbp!]
    \subfigure{\label{fig:ex21}\includegraphics[width=60mm]{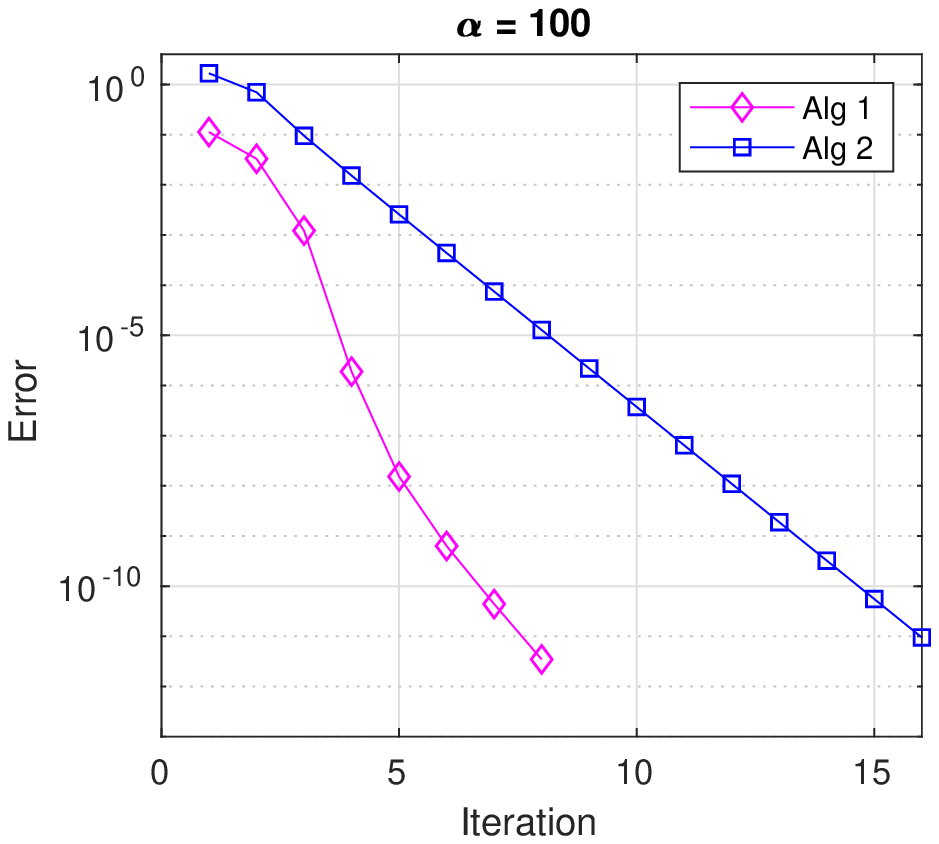}}
    \subfigure{\label{fig:ex22}\includegraphics[width=60mm]{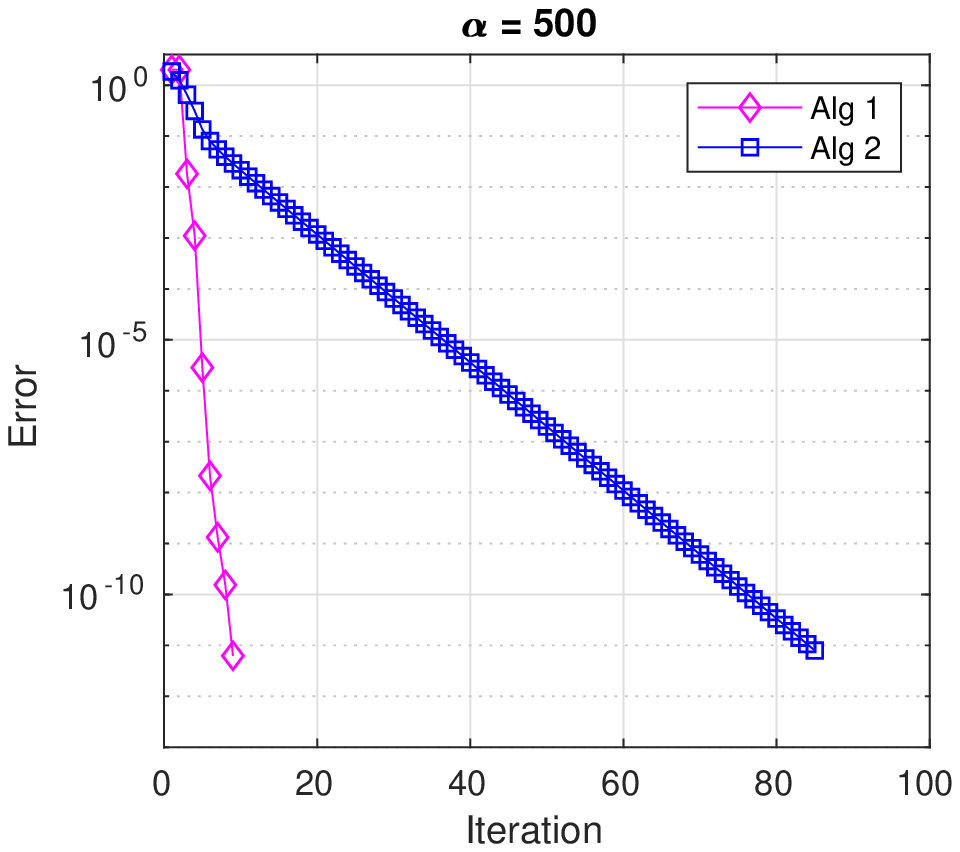}}
     \caption{Algorithm 1 and Algorithm 2 for different $\alpha$ applied t}
        \label{fig:ex3}
\end{figure}
In Figure~\ref{fig:ex3}, we observe that Algorithm~\ref{alg:aversion} converges linearly and Algorithm~\ref{alg:jversion} has much faster convergence, with an initial quadratic phase. The number of iterations required for convergence increases with increase in $\alpha$, as expected from the single step analysis of section~\ref{subsec:singlestep}. The initial quadratic phase is suceeded by an asymptotic slowdown, which can be attributed to the inexact solution of the update equation \eqref{eq:qn1finale}.

\section{Conclusions and Outlook}
This paper shows that taking an inexact Newton approach towards deriving algorithms for problems with eigenvector nonlinearities leads to new algorithmic insights. Using this approach, we derive two algorithms. Algorithm~\ref{alg:aversion} is shown to be the widely used SCF algorithm. This result shows a connection between Newton's method and the SCF algorithm which was previously unknown. Algorithm~\ref{alg:jversion} is a new algorithm, to the best of our knowledge.

We prove that Algorithm~\ref{alg:jversion} exhibits quadratic local convergence. Both Algorithm~\ref{alg:jversion} and Algorithm~\ref{alg:aversion} have favourable convergence properties for problems that are close to being linear, as shown by the single step analysis of Section~\ref{subsec:singlestep}. Numerical simulations for the Gross-Pitaevskii equation in Section~\ref{subsec:gpe} show that Algorithm~\ref{alg:jversion} is a competitive algorithm for the $p = 1$ case. The $p>1$ example in Section~\ref{subsec:invarsubs} shows that Algorithm~\ref{alg:jversion} converges faster than Algorithm~\ref{alg:aversion} even when we solve the update equation \eqref{eq:qn1finale} inexactly.

There are several improvements of the SCF algorithm. Some of these techniques may be interpretable from an implicit viewpoint as well. For instance, acceleration schemes such as DIIS \cite{Rohwedder:2011:DIIS} might be seen as an inexact Newton algorithm. This could be combined with other convergence theory to gain further understanding of DIIS. Another direction that can be explored is to develop application-specific strategies to solve the \eqref{eq:qn1finale} for $p>1$ or approximate solution techniques that lead to superlinear convergence. 
\appendix
\section{Proofs}\label{sec:convproof}
\subsection{Proof of Lemma 4}
\begin{proof}
  
%


Let $F'_k$ be the Jacobian of $F$ evaluated in the iterate $k$.
In the notation of \cite{Dembo:1982:INEXACT} we introduce
a residual denoted $r_k$, corresponding to the difference between a Newton step and
a inexact Newton step:
\begin{eqnarray}\label{eq:newtnewt2}
F'_k\begin{bmatrix}\Delta v^{(k)}\\ \Delta s^{(k)}\end{bmatrix} &=& -F^{(k)}+r_k.
\end{eqnarray}
For notational convenience we now define 
\[
\alpha_k:=\frac{d}{dv}\vect(V^TV)_{V=V^{(k)}}
\]
such that the $Z_k$ in the  Jacobian \eqref{eq:newt} is
$Z_k=\frac12(\alpha_k+\alpha_{k+1})$.
Then, subtracting \eqref{eq:qn1} from \eqref{eq:newtnewt2} the residual
becomes 
\begin{equation*}
  r_k =
  \begin{bmatrix}
    0&0\\
    \frac12(-\alpha_k+\alpha_{k+1})&0    
  \end{bmatrix}\begin{bmatrix}\Delta v^{(k)}\\ \Delta s^{(k)} \end{bmatrix}
\end{equation*}
In our setting
\[
\frac12(-\alpha_k+\alpha_{k+1})=
  \mathcal{O}(\|V^{(k+1)}-V^{(k)}\|)\leq \mathcal{O}\left(\|\Delta V^{(k)}\|_F\right) \le \mathcal{O}\left( 
\left\|\begin{bmatrix}
  \Delta v^{(k)}\\
  \Delta s^{(k)}
\end{bmatrix}\right\|\right)
\]
where we used the smoothness of $\frac{d}{dv}\vect(V^TV)$
such that
\begin{equation}\label{eq:rkbound1}
\|r_k\|\le \mathcal{O}(\left\|\begin{bmatrix}
  \Delta v^{(k)}\\
  \Delta s^{(k)}
\end{bmatrix}\right\|^2)
\end{equation}
By the assumption of monotonic convergence, we have,
\begin{equation}\label{eq:vkvst}
\left\|\begin{bmatrix}\Delta v^{(k)}\\ \Delta s^{(k)}\end{bmatrix}\right\| \le 2\left\|\begin{bmatrix}v^{(k)}-v_*\\s^{(k)}-s_*\end{bmatrix}\right\| = \mathcal{O}\left(\left\|\begin{bmatrix}v^{(k)}-v_*\\s^{(k)}-s_*\end{bmatrix}\right\|\right).
\end{equation}
From the implicit function theorem (e.g. the formulation in \cite[Theorem~9.28]{Rudin:1976:PRINCIPLES}) and the assumption about the invertability of the Jacobian at the solution we get that
\begin{equation}\label{eq:ftaylor}
\begin{bmatrix}
  v^{(k)}\\
  s^{(k)}
\end{bmatrix}=
\begin{bmatrix}
  v_*\\
  s_*
\end{bmatrix}+(F'_*)^{-1}F^{(k)}+\mathcal{O}(\|F^{(k)}\|^2).
\end{equation}
where $F_*'$ is the Jacobian of $F$ evaluated in the solution.
The combination of \eqref{eq:rkbound1}, \eqref{eq:vkvst} and \eqref{eq:ftaylor} leads to 
\begin{equation*}
\|r_k\| = \mathcal{O}(\|F^{(k)}\|^2).
\end{equation*} 
By \cite[Theorem~3.3]{Dembo:1982:INEXACT}, the proof is complete.

\end{proof}

\subsection{Proof of Theorem 7}
\begin{proof}
For any sequence of pairs $(V^{(k)},S^{(k)})$ that satisfies \eqref{eq:qn1finale} with the ${V^{k}}^TV^{k} = I_p$, the corresponding vectorized pairs must satisfy \eqref{eq:qn1}. The iterative method dictated by \eqref{eq:qn1} is also an Inexact Newton Method with the Jacobian approximation 
\begin{equation*}
  \begin{bmatrix}
    A_k&\tilde{B}_k\\
    \tilde{C}_k&0
  \end{bmatrix}
\end{equation*}
where $\tilde{B}_k$ is the modified (1,2) block of the Jacobian. Using this approximation, the residual vector is
\begin{equation*}
r_k = \begin{bmatrix}
    0&\tilde{B}_k-B_k\\
    \tilde{C}_k-C_k&0
  \end{bmatrix}\begin{bmatrix}\Delta v^{(k)}\\ \Delta s^{(k)} \end{bmatrix}
\end{equation*}
Note that
\[
\|\tilde{B}_k-B_k\| = \|I_p\otimes (V^{(k+1)}-V^{(k)})\| \leq \mathcal{O}(\|\Delta V^{(k)}\|_F) \le \mathcal{O}\left(\left\|\begin{bmatrix}\Delta v^{(k)}\\ \Delta s^{(k)} \end{bmatrix}\right\|\right). 
\] 
and hence
\[
\left\|\begin{bmatrix}
    0&\tilde{B}_k-B_k\\
    \tilde{C}_k-C_k&0
  \end{bmatrix} \right\| = \max\{\|\tilde{B}_k-B_k\|,\|\tilde{C}_k-C_k\|\} = \mathcal{O}\left(\left\|\begin{bmatrix}\Delta v^{(k)}\\ \Delta s^{(k)} \end{bmatrix}\right\|\right).
\] 

Repeating the final arguement in the proof of Lemma~\ref{thm:newt_mod1}, the proof is complete.
\end{proof}
\section{Derivation of $J$ for a problem type}\label{sec:deriv}
We consider a specific problem with
\begin{equation*}
A(V) = A_0+\sum\limits_{i=1}^{k}\left(c_i^Th(VV^T)d_i\right)A_i
\end{equation*}
where $c_i,d_i \in\mathbb{R}^{n}$ and $A_i\in\mathbb{R}^{n\times n}$, $\forall i\in\{1,\ldots,k\}$.

Using \eqref{eq:jaclhs}, we have
\begin{equation}\label{eq:specjac}
\begin{split}
J(v) &= I_p\kron A(V)+\sum\limits_{i=1}^{k}\Big(\dfrac{d}{dv}\left(I_p\kron(c_i^Th(VV^T)d_i)A_i\right)\hat{v}\Big)_{\hat{v}=v}\\
     &= I_p\kron A(V)+\sum\limits_{i=1}^{k}\Big(\vect(A_i\hat{V})c_i^T\dfrac{d}{dv}(h(VV^T)d_i)\Big)_{\hat{v}=v}.
\end{split}
\end{equation}
\begin{theorem}\label{thm:hgsame}
Suppose $V\in\mathbb{R}^{n\times p}$ has full column rank. For any $d\in\mathbb{R}^n, \exists\delta > 0$ such that 
\begin{equation*}
\frac{d}{dv}h(VV^T)d = \frac{d}{dv}g_{\delta}(VV^T)d,
\end{equation*}
where 
\begin{equation*}
g_{\delta}(x) = h(x-\delta).
\end{equation*}
\end{theorem}
\begin{proof}
Since $V$ is full rank, we can diagonalize $VV^T$ as $VV^T = PSP^T$, where
\begin{equation*}
s_{i,i} \begin{cases}>0,&\quad\textrm{if }0\leq i \leq p\\
                    =0,&\quad\textrm{if }p< i\leq n.
          \end{cases}
\end{equation*}
and $P$ is orthogonal. Hence, if $0< \delta_0 < \min\limits_{i\in\{1,\ldots,p\}}s_{i,i}$, we have
\begin{equation}\label{eq:gdelta} 
h(VV^T) = Ph(S)P^T = Pg_{\delta_0}(S)P^T = g_{\delta_0}(PSP^T) =  g_{\delta_0}(VV^T).
\end{equation}
By the lower semicontinuity of the rank function, $\exists r>0$ such that $\norm{\Delta V}<r$ implies
\[
\rank(V+\Delta V) = \rank(V) = p.
\]
Using the continuity of the eigendecomposition, we have
\begin{equation*}
(V+\Delta V)(V+\Delta V)^T = (P+\Delta P)(S+\Delta S)(P+\Delta P)^T
\end{equation*}
where $\norm{\Delta P} = \mathcal{O}(\norm{\Delta V})$ and 
\begin{equation*}
\Delta s_{i,i} =\begin{cases}\mathcal{O}(\norm{\Delta V}),&\quad\textrm{if }0\leq i \leq p\\
                    0,&\quad\textrm{if }p< i\leq n.
          \end{cases}
\end{equation*}
Note that $\norm{\Delta V}$ can be chosen to be arbitrarily small and by \eqref{eq:gdelta}, we can choose $\delta > \mathcal{O}(\|\Delta V\|)$ such that 
\begin{equation*}
h\Big((V+\Delta V)(V+\Delta V)^T \Big)-h(VV^T) = g_{\delta}\Big((V+\Delta V)(V+\Delta V)^T \Big)-g_{\delta}(VV^T).
\end{equation*}
For such $\delta$, we have
\begin{equation*}
\begin{split}
&\lim\limits_{\epsilon\to 0}\frac{h\Big((V+\epsilon\Delta V)(V+\epsilon \Delta V)^T\Big)-h(VV^T)}{\epsilon}\\
 = &\lim\limits_{\epsilon\to 0}\frac{g_{\delta}\Big((V+\epsilon\Delta V)(V+\epsilon \Delta V)^T\Big)-g_{\delta}(VV^T)}{\epsilon}.
\end{split}
\end{equation*}
Hence, the proof if complete.
\end{proof}

\begin{theorem}[Frechet derivative of shifted heavyside function]\label{thm:frechhvy}
Let $L_g(W,E)$ denote the Frechet derivative of $g_{\delta}$ at $W$ and applied to $E$. For any $V\in\mathbb{R}^{n\times p}$ such that $V^TV = I$, we have
\begin{equation*}
L_g(VV^T,E) = (I-VV^T)EVV^T+VV^TE(I-VV^T).
\end{equation*}
\end{theorem}
\begin{proof}
In this proof we let $O(m,n)$ and $N(m,n)$ denote the $m$-by-$n$ zeroes and ones matrices respectively. Note that we can write $VV^T$ as follows.
\begin{equation*}
VV^T = \begin{pmatrix}V& V_r\end{pmatrix}\begin{bmatrix}I_p& O(p,n-p)\\O(n-p,p)& I_{n-p}\end{bmatrix}\begin{pmatrix}V^T\\ V_r^T\end{pmatrix},
\end{equation*}
where $V_r\in \mathbb{R}^{n\times(n-p)}$, $V_r^TV_r = I$ and
\begin{equation*}
VV^T+V_rV_r^T = I.
\end{equation*}
This implies that $Z = \begin{pmatrix}V& V_r\end{pmatrix}$ is orthogonal and $Z^{-1} = Z^T = \begin{pmatrix}V^T\\V_r^T\end{pmatrix}$. 
Using Corollary 3.12 in \cite{HIGHAMFOM} with $\mathcal{D} = \mathbb{R}\setminus \{\frac{1}{2}\}$, we have
\begin{equation}\label{eq:gfrech}
L_g(VV^T,E) = Z\left[(g_{\delta}[\lambda_i,\lambda_j])\circ Z^TEZ\right]Z^T,
\end{equation}
where 
\begin{equation*}
\lambda_i = \begin{cases}1,\quad\quad\textrm{if}\;0\leq i\leq p\\
                         0,\quad\quad\textrm{if}\;p+1\leq i\leq n\end{cases}.
\end{equation*}
Using $g_{\delta}(0) = 0$ and $g_{\delta}(1) = 1$ gives us
\begin{equation*}
g_{\delta}[\lambda_i,\lambda_j] = \begin{bmatrix}O(p,p)& N(p,n-p)\\ N(n-p,p)& O(n-p,n-p)\end{bmatrix}
\end{equation*}
and hence,
\begin{equation}\label{eq:gdivdif2}
\begin{split}
&g_{\delta}[\lambda_i,\lambda_j]\circ Z^TEZ \\=\\ &\begin{bmatrix}O(p,p)& \begin{pmatrix}I_p& O(p,n-p)\end{pmatrix}Z^TEZ\begin{pmatrix}O(p,n-p)\\ I_{n-p}\end{pmatrix}\\
                                                  \begin{pmatrix}O(n-p,p)& I_{n-p}\end{pmatrix}Z^TEZ\begin{pmatrix}I_p\\ O(n-p,p)\end{pmatrix}& O(n-p,n-p)\end{bmatrix}
\end{split}
\end{equation}
Noting that
\begin{eqnarray*}
 \begin{pmatrix}I_p& O(p,n-p)\end{pmatrix}Z^T &=& V^T,\\
 Z\begin{pmatrix}O(p,n-p)\\ I_{n-p}\end{pmatrix} &=& V_r
\end{eqnarray*}
and combining with \eqref{eq:gfrech} and \eqref{eq:gdivdif2} gives us
\begin{equation*}
\begin{split}
L_g(VV^T,E) &= \begin{pmatrix}V& V_r\end{pmatrix}\begin{bmatrix}O(p,p)& V^TEV_r\\
                                                  V_r^TEV& O(n-p,n-p)\end{bmatrix}\begin{pmatrix}V^T\\ V_r^T \end{pmatrix}\\
            &= V_rV_r^TEVV^T+VV^TEV_rV_r^T\\ 
            &= (I-VV^T)EVV^T+VV^TE(I-VV^T).
\end{split}
\end{equation*}
Hence, the proof is complete.
\end{proof}

Using Theorem~\ref{thm:frechhvy} and Theorem~\ref{thm:hgsame}, \eqref{eq:specjac} leads to the fact that for any $V\in\mathbb{R}^{n\times p}$ such that $V^TV = I$,
\begin{equation}\label{eq:jacex2}
\begin{split}
J(v) &= I_p\kron A(V)+\sum\limits_{i=1}^{k}\Big(\vect(A_i\hat{V})c_i^T\dfrac{d}{dv}(h(VV^T)d_i)\Big)_{\hat{v}=v}\\
     &= I_p\kron A(V)+\sum\limits_{i=1}^{k}\Big(\vect(A_i\hat{V})c_i^T\dfrac{d}{dv}(g_{\delta}(VV^T)d_i)\Big)_{\hat{v}=v}\\
     &= I_p\kron A(V)+\sum\limits_{i=1}^{k}\vect(A_iV)c_i^T\left(\frac{d}{dw}g_{\delta}(W)d_i\right)_{W=VV^T}\left(\frac{d}{dv}\vect(VV^T)\right).
\end{split}
\end{equation}
Note that
\begin{equation}\label{eq:dhsimple}
\begin{split}
c^T\frac{d}{dw}(g_{\delta}(W)d)\big|_{W=VV^T} &= \begin{bmatrix}c^TL_g(VV^T,E_{1,1})d& c^TL_g(VV^T,E_{2,1})d& \ldots& c^TL_g(VV^T,E_{n,n})d\end{bmatrix}\\
                          &:= \mathcal{L}_g(V,c,d),
\end{split}
\end{equation}
where $E_{i,j} = e_ie_j^T$. Let $P_{n^2}\in\mathbb{R}^{n^2\times n^2}$ be the shuffle matrix such $\vect(W) = P_{n^2}\vect(W^T)$ for any $W\in\mathbb{R}^{n\times n}$. Then
\begin{equation}\label{eq:shuf}
\begin{split}
\frac{d}{dv}\vect(VV^T) = V\otimes I_n + P_{n^2}\Big(\frac{d}{dv}\vect(v\hat{v}^T)\Big)_{\hat{v}=v} = \Big(I_{n^2}+P_{n^2}\Big)(V\kron I_n).\\
\end{split}
\end{equation}
Combining \eqref{eq:jacex2}, \eqref{eq:dhsimple} and \eqref{eq:shuf} leads to 
\begin{equation*}
J(v) = I_p\kron A(V)+\Bigg(\sum\limits_{i=1}^{k}\vect\Big(A_iV\Big)\mathcal{L}_g\Big(V,c_i,d_i\Big)\Bigg)\Big(I_{n^2}+P_{n^2}\Big)(V\kron I_n)
\end{equation*}
for any $V\in\mathbb{R}^{n\times p}$ such that $V^TV = I$.

\bibliography{eliasbib}
\bibliographystyle{spmpsci}

\end{document}